\documentclass{article}
\usepackage[utf8]{inputenc}
\usepackage{fullpage}
\usepackage{hyperref}
\usepackage{amsmath,amssymb,amsthm}
\usepackage{bbm}
\usepackage{mathtools} 
\usepackage{graphicx}
\usepackage{scrextend}
\usepackage{thm-restate}

\graphicspath{ {images/} }
\usepackage{xcolor} 
\usepackage{enumitem} 
\usepackage{lineno}

\usepackage[capitalize]{cleveref}

\usepackage{url}

\hypersetup{
	colorlinks=true, 
	linkcolor=cyan,
	citecolor=blue,
}

\newtheorem{theorem}{Theorem}
\newtheorem{lemma}[theorem]{Lemma}

\newtheorem{proposition}[theorem]{Proposition}

\newtheorem{conjecture}[theorem]{Conjecture}

\theoremstyle{definition}

\newcommand{\LabelQuote}[2]{\vspace{0.5cm}%
	\refstepcounter{equation}
	\parbox{13cm}{\em #1}\hfill (\theequation)\label{#2}\\[0.5cm]}

\newcommand{\trw}{\operatorname{tw}}
\newcommand{\prw}{\operatorname{pw}}
\newcommand{\brn}{\operatorname{bn}}
\newcommand{\D}{\mathcal{D}}

\newcommand{\keywords}[1]{\textbf{Keywords:} #1}
\newcommand{\MSC}[1]{\textbf{MSC (2020):} #1}

\usepackage{ifpdf}

\allowdisplaybreaks

\title{On the Treewidth of Token and Johnson Graphs\thanks{The authors gratefully acknowledge support from CONACYT FORDECYT-PRONACES/39570/2020, SEP-CONACYT A1-S-8397 and UNAM-PAPIIT IA101423 grants.}}

\author{Ruy Fabila-Monroy\thanks{Departamento de Matem\'aticas, CINVESTAV. \texttt{ruyfabila@math.cinvestav.edu.mx}} \and
	Sergio Gerardo G\'omez-Galicia\footnotemark[2] \thanks{\texttt{sgomez@math.cinvestav.mx}} \and
	C\'esar Hern\'andez-Cruz\thanks{Facultad de Ciencias, Universidad Nacional Aut\'onoma de M\'exico, Mexico City Mexico. \texttt{chc@ciencias.unam.mx}} \and
	Ana Trujillo-Negrete\thanks{Centro de Modelamiento Matemático (CNRS IRL 2807), Universidad de Chile, Santiago, Chile.
		Partially supported by CONACYT(México), Convocatoria 2021 de Estancias Posdoctorales por México en Apoyo por SARS-CoV-2(COVID-19), by ANID/Fondecyt Postdoctorado 3220838, and by ANID Basal Grant CMM 
		FB210005. \texttt{ltrujillo@dim.uchile.cl}}
}

\begin{document}
	

\maketitle

\begin{abstract}
	Let $G$ be a graph on $n$ vertices and $1 \le k \le n$ a fixed integer. 
	The  \textit{$k$-token graph} of  $G$ is the graph $F_k(G)$ whose vertex set consists of all $k$-subsets of the vertex set of $G$, where two vertices $A$ and $B$ are adjacent in $F_k(G)$ whenever their symmetric difference $A\triangle B$ is an edge of $G$. 
	In this
	paper we study the treewidth of $F_k(G)$ when $G$ is  a star, path, or a complete
	graph. We show that  in the first two cases, the treewidth
	is of order $\Theta(n^{k-1})$, and of order $\Theta(n^k)$ in the third case.
	We conjecture that our upper bound for the treewidth of $F_k(K_n)$ is tight. This is particularly relevant since $F_k(K_n)$ is isomorphic to
	the well known Johnson graph $J(n,k)$.
\end{abstract}


\MSC{05C05, 05C75, 05C76}


\section{Introduction}

Let $G$ be a graph on $n$ vertices, and let $1 \le k \le n-1 $ be an integer.
The \emph{$k$-token graph} of $G$, denoted by $F_k(G)$, is the graph whose vertex set consists of all $k$-subsets of vertices of $G$,
where two vertices $A$ and $B$ are adjacent in $F_k(G)$ whenever their symmetric difference $A\triangle B$ is an edge of $G$. 
See Figure~\ref{fig:token-graph} for an example. The name ``token graphs''
comes from the following interpretation. Consider $k$ indistinguishable tokens
placed at the vertices of $G$, at most one token per vertex. Create a new graph
whose vertices are all possible token configurations, and make two adjacent if
one token configuration can be reached from the other by sliding a token along an
edge to an unoccupied vertex. The resulting graph is isomorphic to $F_k(G)$. This
interpretation was given by Fabila-Monroy, Flores-Peñaloza, Huemer, Hurtado,
Urrutia, and Wood~\cite{Token}. However, token graphs have been defined
independently several times and under different names
(see~\cite{double_vertex,audenaert,Token,Johns,rudolph,ktuple}).

\begin{figure}[t]
    \centering
    \includegraphics[width=0.5\textwidth]{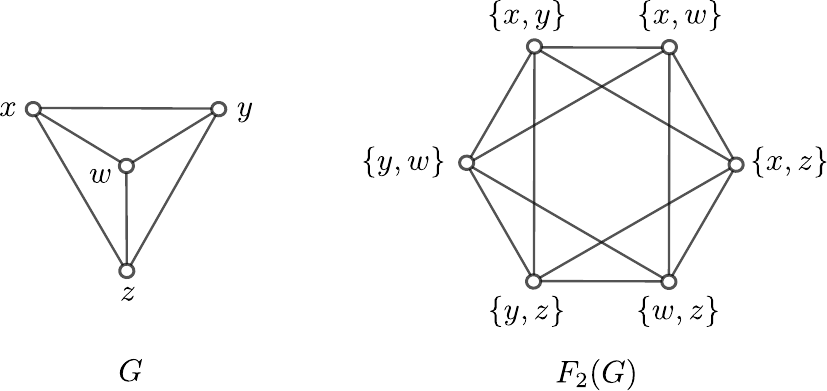}
    \caption{A graph $G$ and its $2$-token graph $F_2(G)$.}
    \label{fig:token-graph}
\end{figure}

A central parameter in structural graph theory is the \emph{treewidth} of a graph, which measures how close a graph is to being a tree. Many $NP$-hard problems become tractable when restricted to graphs of bounded treewidth~\cite{treewidth}. 
Formally, a \emph{tree decomposition} of $G$ is a pair $(T,\mathcal{V})$, where $T$ is a tree and $\mathcal{V} = (V_{t})_{t \in V(T)}$ is a family of sets $V_{t} \subseteq V(G)$, called bags, satisfying the following conditions:
\begin{enumerate}
	\item $V(G)=\bigcup_{t\in V(T)}V_{t}$;
	\item If $uv\in E(G)$, then there exists $t \in V(T)$ such that $u,v \in V_{t}$; and
	\item $V_{t_{1}} \cap V_{t_{3}} \subseteq V_{t_{2}}$ whenever $t_{2}$ is on the path from $t_{1}$ to $t_{3}$ in $T$.
\end{enumerate}
The \emph{width} of a tree decomposition $(T,\mathcal{V})$ is $\max\{|V_t|-1 : t \in V(T)\}$, and the \emph{treewidth} of $G$, denoted $\trw(G)$, is the minimum width over all tree decompositions of $G$ (see Figure \ref{fig:treewidth} for an example). If the tree is restricted to be a path, the decomposition is called a \emph{path decomposition} and the corresponding parameter is the \emph{pathwidth} of $G$, denoted by $\prw(G)$. It is clear from the definitions that $\prw(G)\geq \trw(G)$.
Some basic examples include the facts that the treewidth of a tree is one, while the treewidth of the complete graph on $n$ vertices is $n-1$. Moreover, if $H$ is a subgraph of $G$, then $\trw(H) \le \trw(G)$. 
In this paper, we establish both lower and upper bounds for the treewidth of $F_k(G)$ when $G$ is a star, path, or a complete graph.

For the $k$-token graphs of stars and paths, we establish the following results. 
\begin{restatable}[]{theorem}{stars}
	\label{thm_tw_star}
	Let $S_n$ be the star graph on $n+1$ vertices, and let $1 \le k \le n$ be a fixed integer. Then
	\[
	\trw(F_k(S_n))=\Theta(n^{k-1}).
	\]
\end{restatable} 

\begin{restatable}[]{theorem}{paths}
	\label{proposition_tw_tree}\label{thm_tw_path}
	Let $P_n$ be the path graph on $n$ vertices, and let $1 \le k \le n-1$ be a fixed integer. Then
	\[\trw(F_k(P_n))=\Theta(n^{k-1}). \]
\end{restatable} 

Furthermore, we determine the exact values of $\trw(F_2(P_n))$ and $\trw(F_2(S_n))$. Specifically, we show that  
\[
\textrm{$\trw(F_2(S_n))=n-1$ \quad and \quad $\trw(F_2(P_n))=\left\lfloor \frac{n}{2}\right\rfloor$.} \]

We point out that the $k$-token graph of the complete graph $K_n$ on $n$ vertices
is the well-known Johnson graph $J(n,k)$; this graph is defined as the graph whose vertices
are all the $k$-subsets of an $n$-set, and where two such subsets are adjacent whenever they
intersect in exactly $k-1$ elements. In this paper, we establish the exact value of $\trw(F_2(K_n))$, which is equal to its pathwidth. 

\begin{restatable}[]{theorem}{johnsontwo}
	\label{thm_tw_2-johnson}
	For $n\ge 4$, 
	\begin{displaymath}
		\prw(F_2(K_n))=\trw(F_2(K_n))=\begin{cases}
			\tfrac{n}{2}\left(\tfrac{n}{2}-1\right)+n-2 & \text{when $n$ is
				even,}\\
			\left(\tfrac{n-1}{2}\right)^2+n-2 & \text{when $n$ is odd.}
		\end{cases}
	\end{displaymath}
\end{restatable}
The treewidth of $F_2(K_n)$ has been computed previously by 
Harvey and Wood~\cite{linegraph}; however, we use a different method
based on a result on extremal graph theory, which may be of independent interest. 

\begin{figure}[t]
	\centering
	\includegraphics[width=0.75\textwidth]{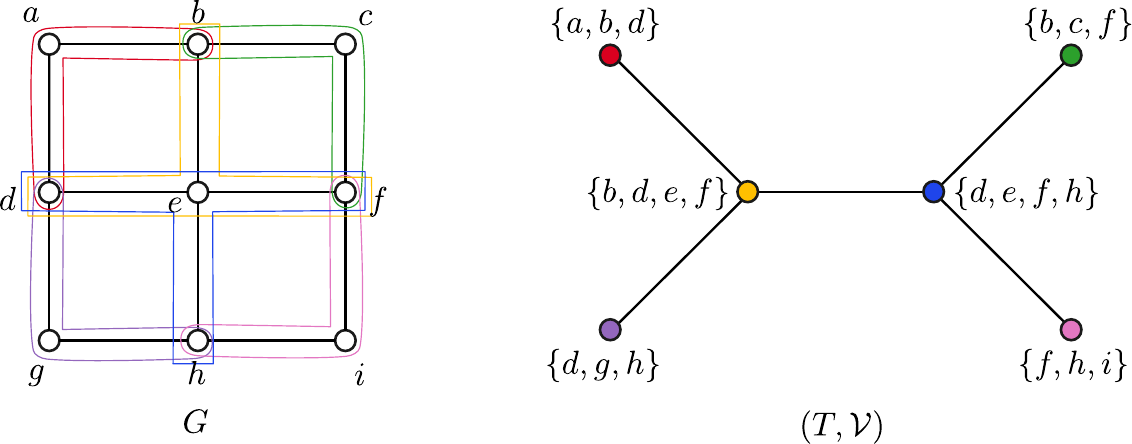}
	\caption{A graph $G$ (left) and an optimal tree decomposition $(T, \mathcal{V})$ (right). The colors in the decomposition indicate the bags that contain the respective vertices in $G$, ensuring that $\trw(G) = 3$.}
	\label{fig:treewidth}
\end{figure}

We also establish lower and upper bounds for 
$\trw(F_k(K_n))$ when $k\ge 3$. 
For the upper bound, we prove the following result:
\begin{restatable}[]{theorem}{johnsonk}
	\label{Theo:twKn}
	\label{thm_tw_k-johnson}
	$\trw(F_k(K_n))$ is at most
	\begin{equation*}
	\begin{aligned}
		\max & \left \{ \left \lfloor\frac{n+1}{k} \right \rfloor\binom{n-\lfloor\frac{n+1}{k}\rfloor}{k-1}+\left \lfloor\frac{n+1}{k} \right \rfloor\sum\limits_{i=2}^{k}\binom{(k-i)\lfloor\frac{n+1}{k}\rfloor}{k-i}-\sum\limits_{i=1}^{k-1}\binom{(k-i-1)\lfloor\frac{n+1}{k}\rfloor+1}{k-i}, \right.  \\
		&  \left. \left \lceil \frac{n}{k} \right \rceil \binom{n-\lceil \frac{n}{k} \rceil}{k-1}+\left \lceil \frac{n}{k} \right \rceil\sum\limits_{i=2}^{k}\binom{(k-i)\lceil \frac{n}{k} \rceil}{k-i}-\sum\limits_{i=1}^{k-1}\binom{(k-i-1)\lceil \frac{n}{k} \rceil+1}{k-i} \right \}-1.
	\end{aligned}
	\end{equation*}
	Further, if $k$ divides $n$ or $n+1$, these two values are equal.
\end{restatable}

For the case $k=3$, we provide the following upper bound, which is consistent with Theorem~\ref{thm_tw_k-johnson}. 
\begin{restatable}[]{corollary}{johnsonthree}
	\label{cor_tw_3-johnson}
	\[\trw(F_3(K_n))\le \left\lceil\frac{n}{3}\right\rceil \binom{n-\left\lceil\frac{n}{3}\right\rceil}{2}+\frac{1}{2}\left\lceil\frac{n}{3}\right\rceil\left(\left\lceil\frac{n}{3}\right\rceil+1\right)-2.\]
\end{restatable}
It is worth mentioning that through computational experimentation, we have confirmed that the upper bound given in  
Corollary~\ref{cor_tw_3-johnson} is tight for certain small values of $n$.

This paper is organized as follows. 
In Section~\ref{section3}, we present the proofs of Theorems~\ref{thm_tw_star} and~\ref{thm_tw_path}, concerning the treewidth of token graphs of stars and paths. 
Section~\ref{sec:f2kn} contains the proof of Theorem~\ref{thm_tw_2-johnson}, which establishes the exact treewidth and pathwidth of $F_2(K_n)$. 
In Section~\ref{sec:fkkn}, we address Theorem~\ref{thm_tw_k-johnson} and Corollary~\ref{cor_tw_3-johnson}, concerning upper bounds for the treewidth of Johnson graphs. 
Finally, in Section~\ref{sec:SR&C}, we present concluding remarks and propose two conjectures.

To prove our results, we make use of fundamental tools from structural and spectral graph theory. These include brambles and their connection to treewidth, the algebraic connectivity of graphs, and spectral properties of token graphs. 
We briefly recall the necessary background and notation. For standard terminology, we refer the reader to~\cite{bondy}. For a finite set $X$, we denote by $\binom{X}{k}$ the family of all $k$-subsets of $X$. Throughout the paper, we assume $V(G) = [n] := \{1, \dots, n\}$ unless stated otherwise, so $V(F_k(G)) = \binom{[n]}{k}$. Note that $F_k(G) \cong F_{n-k}(G)$; thus, we may often assume $1 \leq k \leq \lfloor n/2 \rfloor$.

We say that two subsets of $V(G)$ \emph{touch} if either they have a vertex in
common, or there is an edge between them in $G$. A \emph{bramble},
$\mathcal{B}$, is a family of pairwise touching subsets of vertices of $G$, each
of which induces a connected subgraph. A subset $S$ of vertices of $G$ is said
to \emph{cover} $\mathcal{B}$, if every set of $\mathcal{B}$ contains a vertex
in $S$; we call $S$ a \emph{hitting set} for $\mathcal{B}$. The order of
$\mathcal{B}$, denoted $||\mathcal{B}||$, is the least number of vertices in a
hitting set of $\mathcal{B}$. The \emph{bramble number} of $G$, denoted by 
$\brn(G)$, is the maximum order of a bramble of $G$.
Seymour and Thomas found the following duality theorem, which relates brambles and treewidth.
\begin{theorem}[\cite{Seymour_Thomas}]\label{thm:bramble}
	\label{thm:tw-bramble}
	For any graph $G$, \[\trw(G)=\brn(G)-1.\]
\end{theorem}

We also use a spectral approach based on the algebraic connectivity of a graph. Let $L(G) = D(G) - A(G)$ be the Laplacian matrix of $G$, where $D(G)$ is the degree matrix and $A(G)$ the adjacency matrix. The second smallest eigenvalue of $L(G)$, denoted $\lambda_2(G)$, is called the \emph{algebraic connectivity}. Chandran and Subramanian
showed the following relationship between algebraic connectivity and treewidth
(see also \cite{Amini}). 

\begin{theorem}[\cite{Chandran}]
	\label{chandran} 
	If $G$ is a connected graph with maximum degree $\Delta$, then
	\[
	\frac{|V|}{12\Delta} \lambda_{2}(G)-1 \leq \trw(G).
	\]
\end{theorem}

The Laplacian spectrum of token graphs was studied by Dalfó, Duque,
Fabila-Monroy, Fiol, Huemer, Trujillo-Negrete, and Zaragoza-Martínez
\cite{Laplacian}. The authors conjectured  that the equality 
\begin{equation}
	\label{eq_algconn}
	\lambda_{2}(F_{k}(G))=\lambda_{2}(G)
\end{equation}
always holds; and proved it for certain families of graphs. 
However, unbeknownst to the authors this is a special case of 
Aldous' spectral gap conjecture. This conjecture was proved by 
Caputo, Liggett and Richthammer \cite{caputo}. 
We use (\ref{eq_algconn}) and Theorem \ref{chandran} to obtain lower bounds for the
treewidth of the token graph of paths and complete graphs.

\section{Stars and Paths}{\label{section3}}

The star, $S_n$, of $n$ leaves is the complete bipartite graph $K_{1,n}$. 
Throughout this section, we assume that $V(S_n)=\{0\} \cup[n]$, where vertex
$0$ is the center of $S_n$. We also use $P_n$ to denote the path $(1,\dots,n)$.

Our first result is for the treewidth of token graphs of stars, namely Theorem~\ref{thm_tw_star}:  
\stars*
\begin{proof}
	As mentioned in the Introduction, we have $F_k(S_n)\cong F_{n+1-k}(S_n)$, so without loss of generality, we may assume $k\le \frac{n+1}{2}$.  In what follows, we show the following bounds:  
	\[
		\frac{1}{12} \cdot \frac{n^{k-1}}{k!} + \Theta(n^{k-2}) \leq
		\trw(F_{k}(S_{n})) \leq \binom{n}{k-1}-1= k \cdot \frac{n^{k-1}}{k!} + \Theta(n^{k-2}).
		\] This directly implies that $\trw(F_k(S_n))=\Theta(n^{k-1})$, as desired.

	The graph $F_k(S_n)$ is semiregular (see Lemma~3.8(i) in~\cite{barik}): its vertices have degree either $k$ or $n-(k-1)$, depending on whether they contain the central vertex $0$. Since we are assuming $k\le \frac{n+1}{2}$, the maximum degree of $F_k(S_n)$ is $\Delta(F_k(S_n))=n-(k-1)$. Additionally, it is well known that $\lambda_{2}(S_{n})=1$ (see \cite{oldandnew}). Therefore,
	by (\ref{eq_algconn}) and Theorem \ref{chandran}, we obtain
	\begin{align*}
		\trw(F_{k}(S_{n}))
		&\ge \frac{1}{12(n-(k-1))}\binom{n+1}{k}-1 \\ 
		&=  \frac{1}{12} \cdot \frac{n^{k-1}}{k!}+\Theta(n^{k-2}).
	\end{align*}
	
	To prove the upper bound we construct a tree decomposition $(T,\mathcal{V})$
	of $F_{k}(S_{n})$. Let $T$ be a star graph with $0$ as its center
	vertex, and such that every $A \in \binom{[n]}{k}$ is a
	leaf of $T$. That is, $V(T) = \{0\} \cup \binom{[n]}{k}$, and each vertex $A \in \binom{[n]}{k}$ is adjacent to the center $0$. Let 
	\[
	V_{0} := \left\{ B \in V(F_k(S_n)) :0\in B \right\}
	\textrm{\quad  and \quad }
	V_A:=\{A\}\cup \left \{\left(A \setminus \{a_i\}\right)  \cup \{0\}: a_i \in A \right \},
	\]
	for every $A \in \binom{[n]}{k}$.  Finally, let $\mathcal{V}:=(V_t)_{t \in V(T)}$. 
    See Figure~\ref{fig:stargraph} for an example of this construction. 
	We have the following.
	\begin{enumerate}
		\item Every vertex of $F_{k}(S_{n})$ is contained in some $V_{A}$.
		\item Each edge in $F_k(S_n)$ corresponds to a pair of sets of the form $A$ and $(A \setminus {a_i}) \cup {0}$, where $A \in \binom{[n]}{k}$ and $a_i \in A$.
        Thus, \[A, (A\setminus \{a_i\}) \cup \{0\} \in V_A.\]

		\item For every pair $A, A' \in \binom{[n]}{k}$,  if $|A\cap A'|
		\leq k-2$, then $V_{A} \cap V_{A'} = \emptyset \subset V_{0}$, and
		if $|A\cap A'| = k-1$, then
		\[
		V_{A} \cap V_{A'} = \{(A \cap A')\cup \{0\}\}
		\subset V_{0}.
		\]
	\end{enumerate}
	Therefore, $(T,\mathcal{V})$ is a tree decomposition of $F_k(S_n)$. Its width is
	equal to
	\[
	\binom{n}{k-1}-1 = \frac{1}{(k-1)!} n^{k-1} + \Theta(n^{k-2}).
	\]
\end{proof}

\begin{figure}
	\centering
	\includegraphics[width=0.9\linewidth]{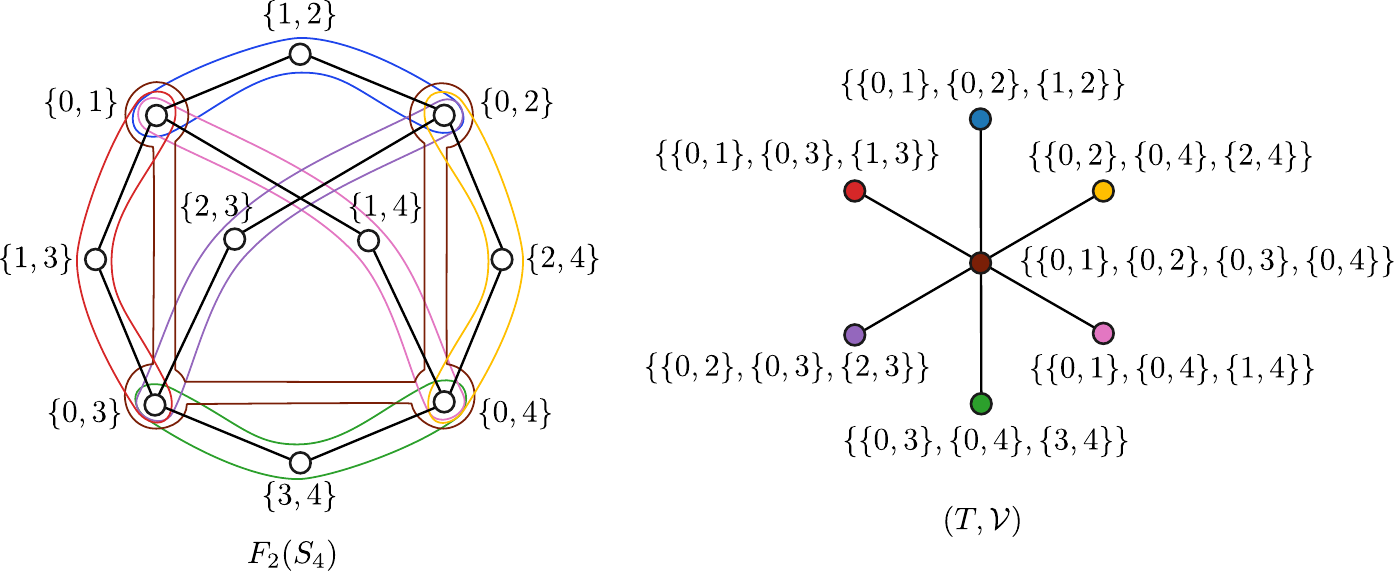}
	\caption{The $2$-token graph $F_2(S_4)$ of the star $S_4$ (left) and an optimal tree decomposition $(T, \mathcal{V})$ (right). The colors in the decomposition indicate the bags that contain the respective vertices in $F_2(S_4)$. }  
	\label{fig:stargraph}
\end{figure}

We now compute the exact value of the treewidth of the $2$-token graph of a
star. 
\begin{proposition}
	\label{pro:bramble_star}
	\[
	\trw(F_{2}(S_{n})) = n-1.
	\]
\end{proposition} 

\begin{proof}
	As shown in the proof of Theorem~\ref{thm_tw_star}, we have 
	\[\trw(F_{2}(S_{n})) \le \binom{n}{1}-1=n-1.\]
	For the lower bound, we construct
	a bramble of $F_2(S_n)$ in which each hitting set has at least $n$ vertices.   For every $1 \le
	i \le n$, let 
	\[
	B_i := \{ \{s,i\} :\ 0 \le s < i \}.
	\]
	Each $B_i$ induces a connected subgraph of $F_2(S_n)$. Moreover, for every $1 \le i < j \le n$, the pair $\{
	\{0,i\}, \{i,j\} \}$ is an edge of $F_2(S_n)$ with one endpoint in  $B_i$ and the other in $B_j$, so the sets $B_i$ and $B_j$ are touching. Thus, the collection
	\[
	\mathcal{B} := \{ B_i \colon 1 \le i \le n \},
	\]
	is a bramble of $F_2(S_n)$. Since $B_i \cap B_j =\emptyset$, any hitting
	set of $\mathcal{B}$ must contain at least one vertex from each $B_i$, and therefore must have cardinality at least $n$. By Theorem
	\ref{thm:tw-bramble}, we conclude that
	\[
	\trw(F_2(S_n)) \ge n-1.
	\]
\end{proof}

We now turn our attention to token graphs of paths. To analyze them, we use the notion of Cartesian products. 
Let $G_1, \dots, G_n$ be graphs. The \textit{Cartesian product} of $G_1, \dots, G_n$ is the graph $G_1 \square\cdots\square G_n$ with vertex set $V(G_1) \times \cdots \times V(G_n)$, where $(x_1, \dots, x_n)$ is adjacent to $(y_1, \dots, y_n)$ if and only if there exists an index $1 \leq i \leq n$ such that $x_i$ is adjacent to $y_i$ and $x_j = y_j$ for all $j \neq i$.

\begin{lemma}
	\label{lem:cartesian}
	If $G_1, \dots, G_k$ are vertex-disjoint subgraphs of $G$, then
	\[
	\trw(F_k(G)) \ge \trw(G_1 \square \cdots \square G_k).
	\]
\end{lemma}
\begin{proof}
	Consider the subgraph of $F_k(G)$ induced by the token configurations where exactly one token is placed in each $G_i$. This subgraph is isomorphic to $G_1 \square \cdots \square G_k$, which implies the desired inequality. 
\end{proof}

To obtain a lower bound for the treewidth of $F_{k}(P_{n})$, we use the following
result of Hickingbotham and Wood~\cite{Structural} regarding the treewidth of
$d$-dimensional grids. 

\begin{theorem}[\cite{Structural}]
	\label{structural}
	Let $d \ge 2$ be a fixed integer, and let $n_1 \ge \cdots \ge n_d \ge 1$. Then 
	\[
	\trw(P_{n_{1}} \square \cdots \square P_{n_{d}}) = \Theta \left(
	\prod_{j=2}^{d}n_{j} \right).
	\]
\end{theorem}

We are now ready to prove Theorem~\ref{thm_tw_path}, which we restate for completeness.
\paths*
\begin{proof}
	We begin by showing that $\trw(F_k(P_n)) = \Omega(n^{k-1})$.
    Consider a partition of $P_n$ into $k$ vertex-disjoint subpaths $P'_1, \dots, P'_k$, each with $\left\lfloor n/k \right\rfloor$ consecutive vertices. Specifically, for each $i = 1, \dots, k$, define 
    \[
    V(P'_i) := \left\{ (i-1)\left\lfloor \frac{n}{k} \right\rfloor + 1, \dots, i\left\lfloor \frac{n}{k} \right\rfloor \right\}.
    \]
    Let $H$ be the subgraph of $F_k(P_n)$ induced by the vertex set
	\[V(H):=\left\{\{x_1,\dots,x_k\}\in V(F_k(P_n)):x_i\in V(P'_i) \textrm{ for each }i=1,\dots,k\right\}.\]
	It is readily seen that $H\cong P'_1\square\dots\square P'_k$. Since $k$ is fixed, 
	Lemma~\ref{lem:cartesian} and
	Theorem~\ref{structural} imply
	\[\trw(F_k(P_n))= \Omega\left(\left\lfloor \frac{n}{k}\right\rfloor ^{k-1}\right)=\Omega\left(n^{k-1}\right).\]
	
	Next, to establish the upper bound, we argue that $\trw(F_k(P_n)) = O(n^{k-1})$. Define the graph $H'$ as the Cartesian product of $k$ copies of $P_{n-(k-1)}$, with $V(P_{n-(k-1)}):=\{1,\dots,n-k+1\}$. 
	Now, define the map $f:V(F_k(P_n))\to V(H')$ as follows. For each vertex $A = \{x_1, \dots, x_k\}\in V(F_k(P_n))$, with $x_1 < \cdots < x_k$, define 
	\[f(A):=(y_1,\dots,y_k), \textrm{ where }y_i=x_i-(i-1) \textrm{ for each }i=1,\dots,k. \]
	
	We will show that $f$ is an isomorphism. It is readily seen that $f$ is a bijection, so it remains to show that $f$ preserves adjacencies and non-adjacencies. Fix a vertex $A=\{x_1,\dots,x_k\}$ in $F_k(P_n)$, and let $B$ be another vertex of $F_k(P_n)$. Observe that $A$ and $B$ are adjacent in $F_k(P_n)$ if and only if there exists $i\in [k]$ such that $B=(A\setminus \{x_i\})\cup \{x'_i\}$, where  $x'_i=x_i-1$ or $x'_i=x_i+1$. This adjacency condition holds if and only if  $f(A)(j)=f(B)(j)$ for all $j\ne i$, whereas $f(A)(i)=x_i-(i-1)$ and $f(B)(i)=x'_i-(i-1)$, meaning that $f(A)(i)$ and $f(B)(i)$ are adjacent in the $i$-th copy of $P_{n-(k-1)}$. This occurs precisely when $f(A)$ and $f(B)$ are adjacent in $H'$. Thus, $AB$ is an edge of $F_k(P_n)$ if and only if $f(A)f(B)$ is an edge of $H'$. Hence, $f$ is an isomorphism.  
	
	Finally, since $k$ is fixed, Theorem~\ref{structural} implies that
	\[\trw(F_k(P_n))=O\left((n-(k-1))^{k-1} \right)=O(n^{k-1}). \]
	The result follows. 
\end{proof}

We now focus on establishing the exact treewidth of $F_2(P_n)$. To this end, we use a combinatorial characterization of treewidth due to Lucena~\cite{lucena}, which is based on vertex orderings and \textit{borders}. 
Given a graph $G$ and a subset $S \subseteq V(G)$, the subgraph of $G$ induced by $S$ is denoted by $G[S]$.  The \textit{border} of $S$, denoted by $\beta(S)$, is defined as follows. 
\begin{itemize}
	\item If $G[S]$ is connected, then $\beta(S)=|N(S)\setminus S|$, where
	$N(S)$ denotes the neighborhood of $S$; 
	\item If $G[S]$ is disconnected, then $\beta(S)=\max\limits_{C \in \mathcal{C}} \beta(C)$,
	where $\mathcal{C}$ is the set of components of $G[S]$.
\end{itemize}

An ordering $\pi$ of the vertices of $G$ is a bijection from $[n]$ to $V(G)$. 
Let $\pi$ be a fixed ordering of the vertices of $V(G)$. For $i \in [n]$,
let $T_i:=\{\pi(1),\dots,\pi(i)\}$. The \textit{max border} of $G$ with
respect to $\pi$, denoted by $MB_\pi(G)$, is defined as 
\[
MB_\pi(G):=\max_{i \in [n]} \beta(T_i).
\] 
Let the \textit{minimax border size} of a graph $G$ be given by 
\[
MMB(G):=\min_{\textrm{orderings } \pi} MB_\pi(G).
\]

Lucena proved the following characterization:
\begin{theorem}[\cite{lucena}]
	\label{theorem_minimax} 
	The minimax border size of $G$ equals its treewidth.
\end{theorem}

We now use this result to compute the exact value of the treewidth of the 2-token graph of a path. 

\begin{proposition}
	\label{proposition_minimax_path}
	\[
	\trw(F_{2}(P_{n})) = \left\lfloor\frac{n}{2}\right\rfloor.
	\]
\end{proposition}
\begin{proof}
	We begin by establishing the lower bound. Consider the subgraph $H$ of $F_2(P_n)$ induced by the vertex set 
	\[\left\{\{x_1,x_2\}\in V(F_2(P_n)):1\le x_1\le \left\lfloor \frac{n}{2}\right\rfloor <x_2\le n\right\}.\] It is straightforward to observe that $H$ is isomorphic to $P_{\left\lfloor \frac{n}{2}\right\rfloor}\square P_{\left\lceil \frac{n}{2}\right\rceil}$. By applying  Lemma~\ref{lem:cartesian} and  Theorem~\ref{structural}, we obtain the lower bound: \[\trw(F_2(P_n))\ge \left\lfloor \frac{n}{2}\right\rfloor.\]
	
	To show the upper bound, we define an ordering $\pi$
	on $V(F_2(P_n))$ and apply Theorem \ref{theorem_minimax}.
	For $3 \le i \le 2n-1$, let
	\[
	\mathcal{D}_i := \{\{x_1,x_2\}\in F_2(P_n) \colon x_1+x_2=i\};
	\]
	see Figure~\ref{fig:path-D} for an example. Note that $|\mathcal{D}_i|$ is
	maximized when $i=n+1$, and in such a case, we have
	$|\mathcal{D}_i|=\lfloor\frac{n}{2}\rfloor$.  
	For every pair of vertices $\{x_1,x_2\},\{y_1,y_2\}\in V(F_2(P_n))$ with
	$x_1 < x_2$ and $y_1 < y_2$, we say that
	\[
	\{x_1,x_2\}\le\{y_1,y_2\} \textrm{ whenever } x_1+x_2<y_1+y_2, \textrm{ or } 
	x_1+x_2=y_1+y_2 \textrm{ and  } x_1 <y_1.
	\]
	Let $\pi$ be the corresponding ordering of $V(F_2(P_n))$. Let  $1
	\le r \le \binom{n}{2}$ be an integer. Recall that $T_r$ is the set of the first $r$ vertices in the order $\pi$, i.e., $T_r=\{\pi(1),\dots,\pi(r)\}$. Note that the subgraph
	of $F_2(P_n)$ induced by $T_r$ is connected, and thus, $\beta(T_r)
	= |N(T_r) \setminus T_r|$. 
	
	Let $i$ be the maximum index such that $\D_1 \cup \dots \cup
	\D_i \subseteq T_r$. 
	\begin{itemize}
		\item If $\pi(r) \in \D_i$, then $\beta(T_r) = |N(T_r)\setminus T_r| =
		|\mathcal{D}_{i+1}|$ and $\beta(T_r) \le \lfloor \frac{n}{2} \rfloor$.
		
		\item Suppose that $\pi(r) \in \D_{i+1}$; thus, $\pi(r) = \{j,i+1-j\}$ for
		some $j$. If $i < n+1$, then $|(N(T_r)\setminus T_r) \cap \mathcal{D}_{i+2}| = j+1$
		and $|(N(T_r)\setminus T_r) \cap \mathcal{D}_{i+1}|=|\mathcal{D}_{i+1}|-j < \lfloor
		\frac{n}{2} \rfloor-j$. Thus, $\beta(T_r) \le \lfloor \frac{n}{2} \rfloor$.
		Suppose that $i \ge n+1$. We have that $|(N(T_r)\setminus T_r) \cap \mathcal{D}_{i+2}|
		= j$ and $|(N(T_r)\setminus T_r) \cap \mathcal{D}_{i+1}| = |\D_{i+1}|-j \le \lfloor
		\frac{n}{2} \rfloor - j$. Thus, $\beta(T_r) \le \lfloor \frac{n}{2}
		\rfloor$.
	\end{itemize}
	Therefore, \[\beta(T_r) \le \left \lfloor \frac{n}{2} \right  \rfloor.\]
	By Theorem \ref{theorem_minimax} we have that
	\[
	\trw(F_2(P_n)) \le MB_\pi(F_2(P_n)) = \max_{i \in [n]} \beta(T_i)= \left\lfloor
	\frac{n}{2} \right\rfloor.
	\]
\end{proof}

\begin{figure}[htb!]
	\centering
	\includegraphics[width=0.4\textwidth]{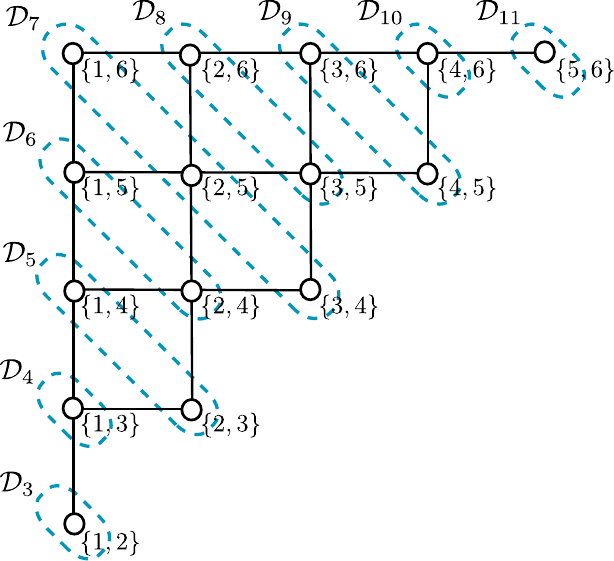}
	\caption{Subsets $\mathcal{D}_i$ in $F_2(P_6)$.}	
	\label{fig:path-D}
\end{figure}

\section{Treewidth of $F_{2}(K_{n})$}{\label{sec:f2kn}}

In this section, we aim to establish Theorem~\ref{thm_tw_2-johnson}. We first provide an upper bound by constructing a path decomposition $(P,\mathcal{V})$ of $F_2(K_n)$ whose
width matches the value stated in Theorem~\ref{thm_tw_2-johnson}. 
For the lower bound, we exhibit a bramble $\mathcal{B}$ of $F_2(K_n)$
whose order is exactly the value given in Theorem~\ref{thm_tw_2-johnson} plus one.  For the reader's convenience, we restate Theorem~\ref{thm_tw_2-johnson} below. 
\johnsontwo*

As mentioned in the Introduction, the treewidth of $F_2(K_n)$ was previously computed by Harvey and Wood~\cite{linegraph} for the line graph of  complete graphs, which is isomorphic to $F_2(K_n)$, and also corresponds to the Johnson graph $J(n,2)$.  In our proof of Theorem~\ref{thm_tw_2-johnson}, however,  we construct a bramble different from the one presented in~\cite{linegraph}, 
with the order of our bramble determined by a result from Faudree and Schelp on extremal graph theory. 

Throughout this section we assume that whenever $\{i,j\}$ is a vertex of $F_2(K_n)$,
we have that $i < j$. We construct the path decomposition $(P,\mathcal{V})$ of $F_2(K_n)$ as follows.  
Let $P$ be the path with vertex set $[n]$, where each vertex $l$ is adjacent to $l+1$ for $1 \le l < n$.
For every $1 \le l \le n$, let 
\[V_l:=\{\{i,j\} \in V(F_2(K_n)): i \le l \le j\}.\]
Let \[\mathcal{V}:=(V_l)_{l \in V(P)}.\]
We have the following.
\begin{enumerate}
	\item For every $\{i,j\}\in V(F_2(K_n))$, we have that $\{i,j\} \in V_j.$ Thus, \[V(F_2(K_n))=\bigcup_{l\in V(P)} V_l.\]
	
	\item If $\{i,j\}\{i',j'\}$ is an edge of $F_2(K_n)$, and  $l$ is the only element 
	in $\{i,j\} \cap \{i',j'\}$, then 
	\[\{i,j\},\{i',j'\}\in V_l.\]
	
	\item  For every $1 \le l_1 < l_3 \le n$ and  $l_1 \le l_2 \le l_3$, if $\{i,j\}$ is a vertex in $V_{l_1} \cap V_{l_3}$, then $i\le l_1<l_3\le j$, which implies $i\le l_2\le j$, and so $\{i,j\}\in V_{l_2}$. Thus, 
	\[V_{l_1} \cap V_{l_3} \subset V_{l_2}.\]
\end{enumerate}
Therefore, $(P,\mathcal{V})$ is a path decomposition of $F_2(K_n)$. 

For every $ 1 \le l \le n$, we have that
\[|V_l|=(l-1)(n-l)+n-1.\]
This is maximized for $l=\lfloor (n+1)/2 \rfloor$. Therefore,
\begin{equation} 
	\label{eq:f2kn}
	\trw(F_2(K_n)) \le \prw(F_2(K_n))\le\begin{cases}
		\tfrac{n}{2}\left(\tfrac{n}{2}-1\right)+n-2 & \text{when $n$ is
			even,}\\
		\left(\tfrac{n-1}{2}\right)^2+n-2 & \text{when $n$ is odd.}
	\end{cases}
\end{equation}

For the lower bound on $\trw(F_2(K_n))$ we make use of brambles and the following result on extremal
graph theory due to Faudree and Schelp~\cite{Faudree}. 
\begin{theorem}[\cite{Faudree}]\label{theorem_extremal}
	If $G$ is a graph with $|V(G)|=kt+r$, $0\le r< k$, containing no path with
	$k+1$ vertices, then $|E(G)|\le t\binom{k}{2}+\binom{r}{2}$.   
\end{theorem}

\begin{proposition}
	\label{prop:brambles-2johnson}
	For $n\ge 4$ we have that
	\[
	\brn(F_2(K_n)) \ge
	\begin{cases}
		\tfrac{n}{2}\left(\tfrac{n}{2}-1\right)+n-1 & \text{when $n$ is even,}\\
		\left(\tfrac{n-1}{2}\right)^2+n-1 & \text{when $n$ is odd.}
	\end{cases}
	\]
\end{proposition}
\begin{proof}
	\ 
	\begin{itemize} 
		\item Suppose that $n$ is odd. Note that the vertices of $F_2(K_n)$ are precisely the edges
		of $K_n$.
		Let 
		\[
		\mathcal{B}:=\left \{B \subset V(F_2(K_n)): B \textrm{ induces a path of  } K_n \textrm{ with }  \frac{n+1}{2} \textrm{ vertices}\right \}.
		\]
		Note that every element of $\mathcal{B}$ induces a connected subgraph of $F_2(K_n)$. 
		Let $B_1, B_2 \in \mathcal{B}$. Since, $ \left | \bigcup_{A \in B_1} A \right |= \left | \bigcup_{A \in B_2} A \right |=\frac{n+1}{2}$,
		we have that $\left (\bigcup_{A \in B_1} A \right )   \cap \left (\bigcup_{A \in B_2}  A \right ) \neq \emptyset$.
		Let $a \in \left (\bigcup_{A \in B_1} A \right )   \cap \left (\bigcup_{A \in B_2}  A \right )$, and let $b_1,b_2 \in V(K_n)$
		such that $\{a,b_1\} \in B_1$ and $\{a,b_2\} \in B_2$. If $b_1=b_2$, then $B_1$ and $B_2$ have a common vertex. If
		$b_1\neq b_2$, then the edge $\{a,b_1\}\{a,b_2\}$  belongs to $F_2(K_n)$ and connects a vertex in $B_1$ to a vertex in $B_2$. Therefore, $B_1$ touches $B_2$, and
		$\mathcal{B}$ is a bramble. 
		
		Let $S$ be a hitting set of $\mathcal{B}$ of minimum cardinality, and let $S^c$ be its
		complement in $V(F_2(K_n))$.  Let $H$ be the subgraph of $K_n$ induced by $S^c$. 
		If $H$ contains a path $P$ with $\tfrac{n+1}{2}$ vertices, then we have a contradiction, since $E(P) \in \mathcal{B}$
		and $E(P) \cap S=\emptyset$.
		Thus, $H$
		contains no path with $\tfrac{n+1}{2}$ vertices. By Theorem
		\ref{theorem_extremal}, with $G=H$,  $k=\tfrac{n-1}{2}$, $t=2$ and
		$r=1$, we have that \[|S^c|=|E(H)|\le 2\binom{\tfrac{n-1}{2}}{2}. \] This
		implies that
		\[
		\brn(F_2(K_n))\ge |S| = \binom{n}{2} - |S^c| \ge \binom{n}{2}-2
		\binom{\tfrac{n-1}{2}}{2} = \left(\tfrac{n-1}{2}\right)^2+n-1,
		\]
		as required. 
		
		\item Suppose that $n$ is even. Let 
		\[V_1:=\{A \in V(F_2(K_n)): n \notin A\} \textrm{ and } V_2:=\{A \in V(F_2(K_n)): n \in A\}.\]
		Let $H_1$ and $H_2$ be the subgraphs of $F_2(K_n)$ induced by $V_1$ and $V_2$, respectively.
		Note that   $H_1 \cong F_2(K_{n-1})$ and $H_2 \cong F_1(K_{n-1}) \cong K_{n-1}$.
        
		Let $\mathcal{B}_1$ be a bramble of $H_1$ defined as in the previous case, and 
		define \[ \mathcal{B}_2:=\left \{B \subset V_2: |B|=\frac{n}{2}  \right \}.\]
		Note that $\mathcal{B}_2$ is a bramble in $H_2$, since $H_2$ is a complete graph.
		Define $\mathcal{B}:=\mathcal{B}_1 \cup \mathcal{B}_2$. 
		Let $B_1,B_2 \in \mathcal{B}$. Since both $\mathcal{B}_1$ and $\mathcal{B}_2$ are
		brambles, the subgraphs induced by $B_1$ and $B_2$ are connected. If
		$B_1$ and $B_2$ are both  in $\mathcal{B}_1$, or are both in  $\mathcal{B}_2$,
		then they touch each other. Now suppose that $B_1 \in \mathcal{B}_1$ and  $B_2 \in \mathcal{B}_2$.
		Note that   $|\bigcup_{A \in B_1} A|=\frac{n}{2}$ and  $|\bigcup_{A \in B_2} A|=\frac{n}{2}+1$.
		Therefore, $\left (\bigcup_{A \in B_1} A \right )   \cap \left (\bigcup_{A \in B_2}  A \right ) \neq \emptyset$.
		Let $a \in \left (\bigcup_{A \in B_1} A \right )   \cap \left (\bigcup_{A \in B_2}  A \right )$, and let $b_1 \in V(K_n)$
		such that $\{a,b_1\} \in B_1$. Since all vertices in $B_2$ contain $n$, we have that $\{a, n\} \in B_2$. Then, the edge $\{a, b_1\}\{a, n\}$ belongs to $F_2(K_n)$ and connects a vertex of $B_1$ to a vertex of $B_2$. Thus, $B_1$ and $B_2$ touch. Therefore, $\mathcal{B}$ is a bramble for $F_2(K_n)$. 
		
		Let $S$ be a hitting set of $\mathcal{B}$ of minimum cardinality; let $S_1:=S \cap V_1$ and $S_2:=S\cap V_2$.
		Note that $S_1$ is a hitting set of $\mathcal{B}_1$ and $S_2$ is a hitting set of $\mathcal{B}_2$. Therefore, 
		\begin{equation}
			\label{eq:first-bramble}
			|S_1| \ge (\tfrac{n-2}{2})^2+n-2.
		\end{equation}
		Suppose that $|S_2| < \frac{n}{2}$. Since every element of $S_2$ contains
		$n$, we have that \[\left | \left (\bigcup_{A \in S_2} A \right ) \setminus \{n\} \right |<\frac{n}{2}.\]
		Thus, there exists a set $B '\subset V(K_n)\setminus \left ( \left (\bigcup_{A \in S_2} A \right ) \setminus \{n\} \right )$
		such that $|B'|=\frac{n}{2}$. Let $B:=\{\{a,n\}:a \in B'\}.$ We have a contradiction since $B \in \mathcal{B}_2$ and $S_2 \cap B =\emptyset$.
		Therefore,
		\begin{equation}
			\label{eq:second-bramble}
			|S_2|\ge \frac{n}{2}.
		\end{equation}
		Combining (\ref{eq:first-bramble})~and~(\ref{eq:second-bramble}),
		we have that
		\[
		\brn(F_2(K_n)) \ge |S| = |S_1| + |S_2|
		\ge (\tfrac{n-2}{2})^2 + n-2 + \tfrac{n}{2} = \tfrac{n}{2}
		(\tfrac{n}{2}-1)+ n-1,
		\]
		as required. 
	\end{itemize}
\end{proof}

Since $\trw(F_2(K_n))=\brn(F_2(K_n))-1$ by Theorem~\ref{thm:bramble}, we derive Theorem~\ref{thm_tw_2-johnson} as a consequence of \eqref{eq:f2kn} and Proposition~\ref{prop:brambles-2johnson}.

\section{Treewidth of $F_{k}(K_{n})$}
\label{sec:fkkn}

In this section, we establish an asymptotic lower bound for $\trw(F_k(K_n))$. Additionally, we generalize the upper bound for $\trw(F_k(K_n))$ presented in Section~\ref{sec:f2kn}, extending it from the case $k=2$ to arbitrary values of $k$.

\begin{proposition}\label{lower_bound_kn}	
	If  $1 \le k \le n-1$ is a fixed integer, then 
	\begin{equation*}
		\trw(F_{k}(K_{n}))\geq \frac{1}{12k(k!)}n^{k}+\Theta(n^{k-1})
	\end{equation*}
\end{proposition}
\begin{proof}
	The graph $F_k(K_n)$ is $k(n-k)$-regular (see~\cite{barik}), so $\Delta(F_k(K_n)) = k(n-k)$. It is well known 
	that $\lambda_{2}(K_{n}) = n$ (see \cite{oldandnew}). Then, from equality
	\eqref{eq_algconn} and Theorem \ref{chandran}, we have     
	\begin{align*}
		\trw(F_{k}(K_{n})) & \ge  \frac{n}{12k(n-k)}\binom{n}{k}-1 \\ 
		&= \left( \frac{n}{12k(n-k)} \right) \left( \frac{n^{k}}{k!}+\Theta(n^{k-1})
		\right)\\
		&= \frac{n^{k}}{12k(k!)}\left(1+\frac{k}{n-k}\right)+\Theta(n^{k-1})\\
		&= \frac{1}{12k(k!)} \cdot n^{k}+\Theta(n^{k-1}).
	\end{align*}
	
\end{proof}

As a direct consequence of Proposition~\ref{lower_bound_kn}, and considering that $\trw(F_k(K_n))$ is trivially bounded above by the number of vertices of $F_k(K_n)$ minus $1$, which is $\binom{n}{k}-1$, we obtain that for a fixed integer $k$, $\trw(F_k(K_n)) = \Theta(n^k)$. In the remainder of this section, however, our aim is to establish a considerably sharper upper bound on $\trw(F_k(K_n))$, as presented in Theorem~\ref{thm_tw_k-johnson}. Notably, this upper bound is exact for $k=2$, and, as shown by computational experimentation, for certain small values of $n$ when $k=3$.

For the remainder of this section, we assume that $2\le k\leq \frac{n}{2}$, as $F_1(K_n)\cong K_n$ and $F_k(K_n)\cong F_{n-k}(K_n)$. 
To obtain the upper bound on $\trw(F_k(K_n))$, we generalize the path decomposition of $F_2(K_n)$ given in
Section~\ref{sec:f2kn}. 

Before proceeding, we give some notation
used throughout this section. 
When considering an $r$-subset $\{a_1,\dots,a_r\}$ with $r\ge 2$, we assume that $a_1<\dots <a_r$. 
Given two $r$-subsets $X=\{x_1,\dots,x_r\}$ and $Y=\{y_1,\dots,y_r\}$, we write $X\le Y$ to mean that
$(x_1,\dots,x_r)\le (y_1,\dots,y_r)$ in the lexicographic order.   
For every $A:=\{a_1,\dots,a_k\} \in \binom{[n]}{k}$, let $A_s=\{a_1,\dots,a_{k-1}\}$ and 
$A_t=\{a_2,\dots,a_k\}$; thus, $A_s,A_t \in \binom{[n]}{k-1}$.

We now define a path decomposition $(P, \mathcal{V})$ of $F_k(K_n)$ as follows:
\begin{itemize} 
	\item Let $P$ be the path with $V(P):=\{X \in \binom{[n-1]}{k-1}\}$, where two subsets $X$ and $Y$ are adjacent if they are consecutive in the lexicographic order of $\binom{[n-1]}{k-1}$. 
	\item For every $X \in \binom{[n]}{k-1}$, let
\[V_X:=\left \{A \in \binom{[n]}{k}: A_s \le X \le A_t\right \}.\] Define $\mathcal{V}:=(V_X)_{X \in V(P)}$. 
\end{itemize}

\begin{theorem}
	$(P,\mathcal{V})$ is a path decomposition of $F_{k}(K_n)$.
\end{theorem}
\begin{proof}
	Let $X \in \binom{[n]}{{k-1}}$ with $n \in X$, and let $Y$ be its predecessor. 
	Let $A \in V_X$. Since $Y \le X$ and $A_s \le X \le A_t$, we have
	that $Y \le A_t$. Given that $n \in X$ and $n \notin A_s$, we have that $A_s < X$.
	The fact that $Y$ is the predecessor of $X$ implies that $A_s \le Y$; thus,
	$A \in V_Y$. Therefore, if $X_1,\dots,X_m=X$ is the sequence of consecutive elements in $\binom{[n]}{k-1}$ in lexicographic order
	such that $n \notin X_1$ and $n \in X_i$ for all $2 \le i \le m$, then 
	\begin{equation} \label{eq:con}
		V_{X}=V_{X_m} \subseteq V_{X_{m-1}}\subseteq \cdots \subseteq V_{X_{1}}.
	\end{equation}
	
	For every $A \in \binom{[n]}{k}$, we have that $A \in V_{A_s}$ and $A_s\in V(P)$.
	Thus, 
	\[V(F_k(G))= \bigcup_{X \in \binom{[n-1]}{k-1}}V_X=\bigcup_{X \in V(P)} V_X.\]
	
	Let $AB$ be an edge of $F_k(G)$. Note that $A\cap B \in \binom{[n]}{k-1}$ and
	$A,B \in V_{A\cap B}$. Thus, by (\ref{eq:con}), 
	\[A,B \in V_X \textrm{ for some vertex } X\in V(P).\]
	
	Let $A \in V(F_k(G))$. Let $X_1,\dots,X_m$ be the sequence
	of consecutive elements in $\binom{[n]}{k-1}$ in lexicographic order such that
	$X_1=A_s$ and $X_m=A_t$. Note that $A \in V_{X_i}$ for all $1 \le i \le m$ and
	these are all the $X \in \binom{[n]}{k-1}$ such that $A \in V_X$. 
	Let $\{i_1,\dots,i_p\}\subseteq [m]$ such that $X_{i_j}\in V(P)$ for each $j\in [p]$, meaning that $n\notin X_{i_j}$. 
	Therefore,  
	the set of vertices $X \in V(P)$ such that $A \in V_X$, which is precisely $\{X_{i_1},\dots ,X_{i_p}\}$, induces a subpath
	of $P$ (since they are consecutive elements of $\binom{[n-1]}{k-1}$ in lexicographic order). 
	This implies that if $Z$ is in the path from $X$ to $Y$ in $P$, then
	\[V_X \cap V_Y \subset V_Z.\]
	Therefore, $(P,\mathcal{V})$ is a path decomposition of $F_{k}(K_n)$.
\end{proof}

We now give the exact cardinality of a bag $V_X \in \mathcal{V}$. 
\begin{proposition} 
	\label{prop:size-of-bags}
	If $X=\{x_1,\dots,x_{k-1}\} \in V(P)$, then 
	\begin{equation*}
		|V_X|=x_1\sum_{i=1}^{k}\binom{n-x_i}{k-i}-\sum_{i=1}^{k-1}\binom{n-x_{i+1}+1}{k-i},
	\end{equation*}
	where for convenience we take $x_k=n$. 
\end{proposition}
\begin{proof}
	Let $A:=\{a_1,a_2,\dots,a_k\} \in V(F_k(K_n))$. Our objective is to determine the number of possibilities for the
	$a_i$'s so that $A \in V_X$. Note that $a_1 \le x_1$ as otherwise $A_s > X$, and $A$ would not be in $V_X$.
	Therefore, $a_1 < x_1$ or $a_1=x_1$.
	\begin{itemize}
		\item Suppose that $a_1<x_1$;  thus, $A_s \le X$ in this case. It remains to establish the 
		conditions so that $X\le A_t$. Note that $X \le A_t$ if and only if either $a_i=x_{i-1}$ for all $2 \le i \le k$, or
		for some $2 \le i \le k$, we have that
		\[a_j=x_{j-1} \textrm{ for all } 2 \le j < i, \textrm{ and } a_i > x_{i-1}.\]
		In the latter case there are $\binom{n-x_{i-1}}{k-(i-1)}$ choices for the remaining
		$\{a_{i},\dots,a_k\}$. We take $\binom{n-x_k}{0}=1$ to represent the possibility
		where $a_i=x_{i-1}$ for all $i=2,\dots,k$; considering that there are $x_1-1$ choices
		for $a_1$ so that $a_1 < x_1$, we have that there in total
		\begin{equation}
			\label{eq:binomial1}
			(x_1-1)\left[\binom{n-x_1}{k-1}+\binom{n-x_2}{k-2}+\dots+
			\binom{n-x_{k}}{0}\right] =(x_1-1)\sum_{i=1}^{k}\binom{n-x_i}{k-i}
		\end{equation}
		possible choices for $\{a_1,\dots,a_k\}$ so that $a_1 < x_1$ and $A \in V_X$.
		
		\item Suppose that $a_1=x_1$; since $a_2 > a_1 =x_1$, we have that $X \le A_t$ in this case.
		It remains to establish the conditions so that $A_s \le X$.
		Note that $A_s \le X$ if and only if either $a_i=x_{i}$ for all $2 \le i \le k-1$, or
		for some $2 \le i \le k-1$, we have that
		\[a_j=x_{j} \textrm{ for all } 2 \le j < i, \textrm{ and } x_{i-1} < a_i < x_{i}.\]
		In the latter case, there are $\binom{n-a_i}{k-i}$ choices for the remaining
		$\{a_{i+1},\dots,a_k\}$. Recall that for convenience we take $x_k=n$. Thus, we have in total
		\begin{equation}
			\label{eq:binomial2}
			\sum_{x_1<a_2<x_2}\binom{n-a_2}{k-2}+\sum_{x_2<a_3<x_3}\binom{n-a_3}{k-3}+\dots
			+\sum_{x_{k-1}<a_{k}< x_{k}}\binom{n-a_{k}}{0} +1
		\end{equation}
		possible choices for $\{a_1,\dots,a_k\}$ so that $a_1 =x_1$ and $A \in V_X$.
		Observe that for $1< i\le k$, we have that
		\begin{align*}
			\sum_{x_{i-1}<a_i<x_i}\binom{n-a_i}{k-i}&=\sum^{n-x_{i-1}-1}_{j=n-x_{i}+1}\binom{j}{k-i}\\
			&=\sum_{j=1}^{n-x_{i-1}-1}\binom{j}{k-i}-\sum_{j=1}^{n-x_{i}}\binom{j}{k-i}\\
			&=\binom{n-x_{i-1}}{k-i+1}-\binom{n-x_i+1}{k-i+1},
		\end{align*}
		where the last equality follows from the hockey-stick identity.\footnote{$\sum_{i =r }^n \binom{i}{r}=\binom{n+1}{r+1}$ for all natural numbers $n\ge r$.}
		In particular, if $x_i=x_{i-1}+1$, we have $\sum\limits_{x_{i-1}<a_i<x_i}\binom{n-a_i}{k-i}=0$. 
		
		Rephrasing Equation~\eqref{eq:binomial2}, we obtain 
		\begin{equation}
			\label{eq:rewriting-2}
			\begin{split}
				1+\sum_{i=2}^k\sum_{x_{i-1}<a_i<x_i}\binom{n-a_i}{k-i}&= 1+\sum_{i=2}^k \left[\binom{n-x_{i-1}}{k-i+1}-
				\binom{n-x_i+1}{k-i+1}\right]\\
				&=1+\sum_{i=1}^{k-1}\left[\binom{n-x_i}{k-i}-\binom{n-x_{i+1}+1}{k-i}\right]\\
				&=\sum_{i=1}^{k}\binom{n-x_i}{k-i}-\sum_{i=1}^{k-1}\binom{n-x_{i+1}+1}{k-i}
			\end{split}
		\end{equation}
		
	\end{itemize} 
	Combining Equations~\eqref{eq:binomial1}~and~\eqref{eq:rewriting-2}, we have that
	\begin{equation*}
		|V_X|=x_1\sum_{i=1}^{k}\binom{n-x_i}{k-i}-\sum_{i=1}^{k-1}\binom{n-x_{i+1}+1}{k-i}.
	\end{equation*}
	This concludes the proof. 
\end{proof}

To determine the bag $V_X$ of maximum size, we partition $\mathcal{V}$
into families, and determine the bag of maximum size for each of these families.  

For each $i\in \{1,\dots,n-k+1\}$, let 
\[\mathcal{V}_i:=\{V_X \in \mathcal{V}: X=\{x_1,\dots,x_{k-1}\} \textrm{ and }x_1=i\}.\] 
We now define a family $\mathcal{M}_i$ of $(k-1)$-subsets of $[n-1]$, which will serve as a subset of indices for bags in $\mathcal{V}_i$. Specifically, let $\mathcal{M}_i$ consist of all sets $X = \{x_1, \dots, x_{k-1}\} \in \binom{[n-1]}{k-1}$ such that $x_1 = i$ and, for each $j = 2, \dots, k-1$, the following conditions hold:
\begin{enumerate}[label=$(\alph*)$]
	\item \label{def-Mi-i}  $x_j=i+j-1$ whenever   $n-(k-j)i+1< i+j-1$, and, 
	\item \label{def-Mi-ii}  $x_j\in \{n-(k-j)i,n-(k-j)i+1\}$, otherwise. 
\end{enumerate}
If for some $2 \le j \le k-2$ we have that $n-(k-j)i+1 \ge i+j-1$, then $n-(k-(j+1))i+1 \ge i+(j+1)-1$. This implies that the sequence
$x_1,\dots,x_{k-1}$ is constructed by first applying rule \ref{def-Mi-i} and afterwards (possibly) only applying rule \ref{def-Mi-ii}. In particular, 
the sequence is increasing and thus defines a valid element of $\binom{[n-1]}{k-1}$.

Therefore, each $X \in \mathcal{M}_i$ indexes a bag $V_X \in \mathcal{V}_i$, and we have
\[
\{V_X : X \in \mathcal{M}_i\} \subseteq \mathcal{V}_i.
\]

\begin{lemma}
	\label{lemma:max-x_i_value}
	  Let  $i\in \{1,\dots,n-k+1\}$ and  $X \in \mathcal{V}_i$. Then 
	  \[|V_X|=\max\limits_{Y\in\mathcal{V}_{i}}|V_Y|\textrm{ if and only if }
	X \in \mathcal{M}_i.\]
\end{lemma}
\begin{proof}
 Let	$X\in \mathcal{V}_{i}$  and let $Y=\{y_1,\dots,y_{k-1}\}\in \mathcal{M}_{i}$ .
 We define the distance between $X$ and $Y$ as
 \[d(X,Y):=\sum_{j=1}^{k-1} |x_j-y_j|.\]
  We iteratively construct a sequence 
	\[X=X_0,X_1,\dots,X_m=Y\]
	of elements in $\mathcal{V}_i$, such that:
	\begin{itemize}
	\item[$(1) $] $|V_{X_0}| \le |V_{X_1}|\le \cdots \le |V_{X_m}|$;
	\item[$(2)$] $d(X_{l},Y)= d(X_{l-1},Y)-1$ for all $2 \le l \le m$; and
	\item[$(3)$] $|V_{X_0}|<|V_{X_m}|$ if $X \notin \mathcal{M}_i$.
	\end{itemize}
	
	Suppose that $X_{l-1}=\{x_1,\dots,x_{k-1}\}$ for some $2 \le l\le m$  is the last element so constructed 
	and that  $X_{l-1} \neq Y$. If $x_j<y_j$ for some $2 \le j \le m$, then let $j^\ast$ be the largest
	index such that $x_{j^\ast}<y_{j^\ast}$. Otherwise,  let $j^\ast$ be the smallest
	index such that $x_{j^\ast} > y_{j^\ast}$. Observe that, in both cases, $j^\ast\ne 1$.

	Let 
	\[
	X_l := 
	\begin{cases} 
		X_{l-1} \setminus \{x_{j^\ast}\} \cup \{x_{j^\ast} + 1\}, & \text{if } x_{j^\ast} < y_{j^\ast}, \\
		X_{l-1} \setminus \{x_{j^\ast}\} \cup \{x_{j^\ast} - 1\}, & \text{if } x_{j^\ast} > y_{j^\ast}.
	\end{cases}
	\]
	Suppose $x_{j^\ast}<y_{j^\ast}$. Since $x_{j^\ast}<y_{j^\ast}\le n-1$, we have $x_{j^\ast}+1\le n-1$. Furthermore, by the choice of $j^\ast$, 
	we have $x_{j^\ast+1}=y_{j^\ast+1}\ge y_{j^\ast}+1>x_{j^\ast}+1$. Therefore, $x_{j^\ast}+1\notin X_{l-1}$, ensuring that $X_l$ is well-defined. Now, suppose that $x_{j^\ast}>y_{j^\ast}$. By the choice of $j^*$, we know that  
	$x_{j^\ast-1}=y_{j^\ast-1}\le y_{j^\ast}-1<x_{j^\ast}-1$. Hence, $x_{j^\ast}-1\notin X_{l-1}$, implying that $X_l$ is well-defined. 
In either case, $X_{l} \in \mathcal{V}_i$ and
\[d(X_{l},Y)= d(X_{l-1},Y)-1. \]
This proves $(2)$. We now show $(1)$. 
	
	Suppose that $x_{j^\ast}<y_{j^\ast}$.
	If $y_{j^\ast}$ was constructed by applying 
	rule \ref{def-Mi-i}, then \[y_{j^\ast}=y_1+j^\ast-1=x_1+j^\ast-1\le x_{j^\ast},\]a contradiction.
 	Therefore, $y_{j^\ast} \ge n-(k-j^\ast)x_1$.
	By Proposition~\ref{prop:size-of-bags}, we have that
	\begin{align*} 
		|V_{X_{l}}|-|V_{X_{l-1}}|&=\left[x_1\sum_{\substack{j=1 \\ j \neq j^\ast}}^{k}\binom{n-x_j}{k-j}-\sum_{\substack{j=1 \\ j \neq j^\ast-1}}^{k-1}\binom{n-x_{j+1}+1}{k-j} + x_1\binom{n-(x_{j^\ast}+1)}{k-(j^\ast-1)} -\binom{n-(x_{j^\ast}+1)+1}{k-(j^\ast-1)}\right ] \\
		&\qquad \qquad -\left[x_1\sum_{\substack{j=1 \\ j \neq j^\ast}}^{k}\binom{n-x_j}{k-j}-\sum_{\substack{j=1 \\ j \neq j^\ast-1}}^{k-1}\binom{n-x_{j+1}+1}{k-j} + x_1\binom{n-x_{j^\ast}}{k-(j^\ast-1)} -\binom{n-x_{j^\ast}+1}{k-(j^\ast-1)}\right ] \\
		&=x_1\left[\binom{n-x_{j^\ast}-1}{k-j^\ast}-\binom{n-x_{j^\ast}}{k-j^\ast}\right]-\left[\binom{n-x_{j^\ast}}{k-{j^\ast}+1}-\binom{n-x_{j^\ast}+1}{k-j^\ast+1}\right]\\
		&=\binom{n-x_{j^\ast}}{k-{j^\ast}}-x_1\binom{n-x_{j^\ast}-1}{k-j^\ast-1}\\
		&=\binom{n-x_{j^\ast}-1}{k-j^\ast-1}\left[\frac{(n-(k-j^\ast)x_1)-x_{j^\ast}}{k-j^\ast}\right]\\
		&  \ge 0.
	\end{align*}
	Where in the last line 
	
	\LabelQuote{we have equality if and only if  $x_{j^\ast}=n-(k-j^\ast)x_1$.}{j-ast-1}
	
	Suppose that $x_{j^\ast} > y_{j^\ast}$. 
	By Proposition~\ref{prop:size-of-bags}, we have that
	\begin{align*} 
		|V_{X_{l}}|-|V_{X_{l-1}}|&=\left[x_1\sum_{\substack{j=1 \\ j \neq j^\ast}}^{k}\binom{n-x_j}{k-j}-\sum_{\substack{j=1 \\ j \neq j^\ast-1}}^{k-1}\binom{n-x_{j+1}+1}{k-j} + x_1\binom{n-(x_{j^\ast}-1)}{k-(j^\ast-1)} -\binom{n-(x_{j^\ast}-1)+1}{k-(j^\ast-1)}\right ] \\
		&\qquad \qquad -\left[x_1\sum_{\substack{j=1 \\ j \neq j^\ast}}^{k}\binom{n-x_j}{k-j}-\sum_{\substack{j=1 \\ j \neq j^\ast-1}}^{k-1}\binom{n-x_{j+1}+1}{k-j} + x_1\binom{n-x_{j^\ast}}{k-(j^\ast-1)} -\binom{n-x_{j^\ast}+1}{k-(j^\ast-1)}\right ] \\
		&=x_1\left[\binom{n-x_{j^\ast}+1}{k-j^\ast}-\binom{n-x_{j^\ast}}{k-j^\ast}\right]-\left[\binom{n-x_{j^\ast}+2}{k-{j^\ast}+1}-\binom{n-x_{j^\ast}+1}{k-j^\ast+1}\right]\\
		& = x_1\binom{n-x_{j^\ast}}{k-j^\ast-1}-\binom{n-x_{j^\ast}+1}{k-j^\ast}\\
		& = x_1\binom{n-x_{j^\ast}}{k-j^\ast-1}-\frac{n-x_{j^\ast}+1}{k-j^\ast}\binom{n-x_{j^\ast}}{k-j^\ast-1}  \\
		& = \binom{n-x_{j^\ast}}{k-j^\ast-1}\left[x_1-\frac{n-x_{j^\ast}+1}{k-j^\ast} \right] \\	
		& = \binom{n-x_{j^\ast}}{k-j^\ast-1}\left[\frac{x_{j^\ast}-\left( n-(k-j^\ast)x_1 +1\right )}{k-j^\ast} \right]\\
		& \ge 0,\\		
	\end{align*}
	Where in the last line 
	
	\LabelQuote{we have equality if and only if  $x_{j^\ast}=n-(k-j^\ast)x_1 +1$.}{j-ast-2}
	This proves $(1)$. 
	
	Suppose that $X \notin \mathcal{M}_i$. Let now $2 \le j^\ast \le k$ be the smallest index
	such that $x_{j^\ast}$ does not satisfy rules \ref{def-Mi-i} and \ref{def-Mi-ii}. Let $2 \le l \le m$ such that
	\[x_{j^\ast} \in X_{l-1} \textrm{ and }  x_{j^\ast} \notin X_{l}.\]
	
	Suppose that $x_{j^\ast} < y_{j^\ast}$. If $y_{j^\ast}$ was constructed by applying 
 	rule \ref{def-Mi-i}, then \[y_{j^\ast}=y_1+j^\ast-1=x_1+j^\ast-1\le x_{j^\ast},\] a contradiction.
 	Thus, $y_{j^\ast}$ was constructed by rule \ref{def-Mi-ii} and  $n-(k-j^\ast)i+1 \ge i+j^\ast-1$.
 	This implies that 	\[x_{j^\ast} < n-(k-j^\ast)x_1.\]
 	By \eqref{j-ast-1}, we have that \[|V_{X_{l}}|>|V_{X_{l-1}}|.\]
 	
 	Suppose that $x_{j^\ast} > y_{j^\ast}$. Suppose that $y_{j^\ast}$ was produced by rule \ref{def-Mi-i}.
 	Thus, \[n-(k-j^\ast)i+1< i+j^\ast-1=y_{j^\ast}<x_{j^\ast}.\]
 	By \eqref{j-ast-2}, we have that \[|V_{X_{l}}|>|V_{X_{l-1}}|.\]
 	Suppose that $y_{j^\ast}$ was produced by rule \ref{def-Mi-ii}. Thus, $n-(k-j^\ast)i+1 \ge i+j^\ast-1$
 	and \[x_{j^\ast} > n-(k-j^\ast)i+1.\]
 	By \eqref{j-ast-2}, we have that \[|V_{X_{l}}|>|V_{X_{l-1}}|.\]
 	This proves $(3)$.
\end{proof}

\begin{lemma}\label{lemma:possible-bags}
	Let $X:=\{x_1,\dots,x_{k-1}\} \in V(P)$ that maximizes $|V_X|$. Then $x_1$ is equal to
	\[\left \lfloor \frac{n+1}{k} \right \rfloor \textrm{ or }  \left \lceil \frac{n}{k} \right \rceil.\]
\end{lemma}
\begin{proof}
Suppose for a contradiction that $x_1 \neq \lfloor \frac{n+1}{k}  \rfloor , \lceil \frac{n}{k} \rceil.$
	We show that there exists $Y \in V(P)$ such that  $|V_Y|>|V_X|$. 
	If $k$ divides $n$ or $n+1$, then  $\lfloor \frac{n+1}{k}  \rfloor = \lceil \frac{n}{k} \rceil$;
	and if $k$ divides neither, then $\lfloor \frac{n+1}{k}  \rfloor =\lfloor \frac{n}{k}\rfloor < \lceil \frac{n}{k} \rceil$;
	therefore, $\lfloor \frac{n+1}{k}  \rfloor \le \lceil \frac{n}{k} \rceil$.
	
	\begin{itemize}
		\item 	Suppose that $x_1<\lfloor\frac{n+1}{k}\rfloor$. Next, we show that $x_2>x_1+1$. Set $i=x_1$ and $j=2$. 
		By Lemma~\ref{lemma:max-x_i_value}, we have $X\in \mathcal{M}_{i}$. 
		If  $n-(k-j)i+1<i+j-1$, then $n-(k-1)i<0$, implying that $i>\frac{n}{k-1}$. 
		Since $i=x_1<\lfloor\frac{n+1}{k}\rfloor$ by hypothesis, this would imply $k>n-1$, which contradicts that $k\le \frac{n}{2}< n-1$. 
		Hence, $x_2$ is selected according to \ref{def-Mi-ii} in the definition of $\mathcal{M}_i$,  which ensures that $x_2\in \{n-(k-2)i,n-(k-2)i+1\}$. Consequently, we obtain $x_2>x_1+1$.   
		
		Now, we show that $n-(k-2)i>i+1$. Suppose otherwise, so $n-1\le (k-1)i$, implying that $i\ge \frac{n-1}{k-1}$. 
		Given that $i=x_1<\lfloor\frac{n+1}{k}\rfloor$ by hypothesis, we would have $k>\frac{n+1}{2}$, a contradiction since $k\le \frac{n}{2}$
		by assumption.  
		Therefore, regardless of the value of $x_2\in \{n-(k-2)i,n-(k-2)i+1\}$, we have $x_2>x_1+1$, as desired.

		Let \[Y:=\{y_1,\dots,y_{k-1}\}=(X\setminus \{x_1\})\cup \{x_1+1\};\]  
		thus, $y_1=x_1+1$ and $y_j=x_j$ for every $j\in \{2,\dots,k-1\}$. 
		Note that this is well-defined because we have already established that $x_2 > x_1 + 1$, ensuring that the choice of $y_1 = x_1 + 1$ is valid.
		By Proposition~\ref{prop:size-of-bags}, we have that
		\begin{align*}
			|V_Y|-|V_X|&= y_1\sum_{i=1}^{k}\binom{n-y_i}{k-i}-\sum_{i=1}^{k-1}\binom{n-y_{i+1}+1}{k-i}\\
			&\qquad -x_1\sum_{i=1}^{k}\binom{n-x_i}{k-i}+\sum_{i=1}^{k-1}\binom{n-x_{i+1}+1}{k-i}\\
			&= y_1\sum_{i=1}^{k}\binom{n-y_i}{k-i}-x_1\sum_{i=1}^{k}\binom{n-x_i}{k-i}\\
			&=(x_1+1)\binom{n-x_1-1}{k-1}+(x_1+1)\sum_{i=2}^{k}\binom{n-x_i}{k-i}\\
			&\qquad \qquad -x_1\binom{n-x_1}{k-1}-x_1\sum_{i=2}^{k}\binom{n-x_i}{k-i}\\
			&=(x_1+1)\binom{n-x_1-1}{k-1}+\sum_{i=2}^{k}\binom{n-x_i}{k-i}-x_1\binom{n-x_1}{k-1}\\
			&>(x_1+1)\binom{n-x_1-1}{k-1}-x_1\binom{n-x_1}{k-1}\\
			&=\frac{(x_1+1)(n-x_1-1)!}{(k-1)!(n-x_1-k)!}-\frac{x_1(n-x_1)!}{(k-1)!(n-x_1-k+1)!}\\
			&=\frac{(n-x_1-1)!}{(k-1)!(n-x_1-k)!}\left[(x_1+1)-\frac{x_1(n-x_1)}{n-x_1-k+1}\right]\\
			&=\binom{n-x_1-1}{k-1}\left[\frac{n+1-k(x_1+1)}{n-x_1-k+1}\right] \\
			&\ge 0.
		\end{align*}
		Where the last inequality is obtained from the following. Since $x_1+1\le \lfloor\frac{n+1}{k}\rfloor$,
		we have that 
		\[n+1-k(x_1+1) \ge n+1-k\left\lfloor \frac{n+1}{k}\right\rfloor \ge n+1-(n+1) = 0. \]
		Recall that $n \notin X$. This implies
		that $x_1 \le x_{k-1}-(k-2) \le n-k+1$ and $n-x_1-k+1 \ge 0$.
		Furthermore, since a strict inequality ($>$) appears in the sequence of steps, this ensures that $|V_Y| > |V_X|$.

		\item Suppose now that $x_1 >\lceil\frac{n}{k}\rceil$. We choose \[Y:=\{y_1,\dots,y_{k-1}\}=(X\setminus \{x_1\})\cup \{x_1-1\};\]
		thus $y_1=x_1-1\ge \lceil \frac{n}{k}\rceil$, and $y_i=x_i$ for $i=2,\dots,k-1$.
		Since $x_i\ge x_1+i-1$, for all $i=2,\dots,k$, we have that
		\begin{align*}
			\binom{n-x_1}{k-2}&=\binom{n-x_1-1}{k-2}+\binom{n-x_1-1}{k-3}\\
			&=\binom{n-x_1-1}{k-2}+\binom{n-x_1-2}{k-3}+\binom{n-x_1-2}{k-4}\\
			&\qquad\qquad\vdots \\
			&=\binom{n-x_1-1}{k-2}+\binom{n-x_1-2}{k-3}+\dots+\binom{n-x_1-k+1}{0}\\
			&=\sum_{i=2}^k\binom{n-(x_1+i-1)}{k-i}\\
			&\ge \sum_{i=2}^k\binom{n-x_i}{k-i}.
		\end{align*}
		Furthermore, if $x_1-1= \lceil \frac{n}{k}\rceil$, we have that $x_2>x_1+1$, and so we obtain that
		\begin{equation}
			\label{eq:x_1greater}
			\binom{n-x_1}{k-2}>\sum_{i=2}^k \binom{n-x_i}{k-i}. 
		\end{equation}
		
		We now compute $|V_X|-|V_Y|$. 
		By Proposition~\ref{prop:size-of-bags}, we have  that 
		\begin{align*}
			|V_X|-|V_Y|&= x_1\sum_{i=1}^{k}\binom{n-x_i}{k-i}-\sum_{i=1}^{k-1}\binom{n-x_{i+1}+1}{k-i}\\
			&\qquad\qquad -y_1\sum_{i=1}^{k}\binom{n-y_i}{k-i}+\sum_{i=1}^{k-1}\binom{n-y_{i+1}+1}{k-i}\\
			&= x_1\sum_{i=1}^{k}\binom{n-x_i}{k-i}-y_1\sum_{i=1}^{k}\binom{n-y_i}{k-i}\\
			&=x_1\binom{n-x_1}{k-1}+x_1\sum_{i=2}^{k}\binom{n-x_i}{k-i}\\
			&\qquad\qquad -(x_1-1)\binom{n-x_1+1}{k-1}-(x_1-1)\sum_{i=2}^{k}\binom{n-x_i}{k-i}\\
			&=x_1\binom{n-x_1}{k-1}+\sum_{i=2}^{k}\binom{n-x_i}{k-i}-(x_1-1)\binom{n-x_1+1}{k-1}\\
			&=x_1\binom{n-x_1}{k-1}-(x_1-1)\left(\binom{n-x_1}{k-1}+\binom{n-x_1}{k-2} \right)+\sum_{i=2}^{k}\binom{n-x_i}{k-i}\\
			&=\binom{n-x_1}{k-1}-(x_1-1)\binom{n-x_1}{k-2}+\sum_{i=2}^{k}\binom{n-x_i}{k-i}\\
			&\le \binom{n-x_1}{k-1}+\binom{n-x_1}{k-2}-(x_1-1)\binom{n-x_1}{k-2}\\
			&=\binom{n-x_1+1}{k-1}-(x_1-1)\binom{n-x_1}{k-2}\\
			&=\frac{n-x_1+1}{k-1}\binom{n-x_1}{k-2}-(x_1-1)\binom{n-x_1}{k-2}\\
			&=\binom{n-x_1}{k-2}\left[\frac{n-x_1+1}{k-1}-(x_1-1)\right]\\
			&=\binom{n-x_1}{k-2}\left[\frac{n-x_1+1-(x_1-1)(k-1)}{k-1}\right]\\
			&=\binom{n-x_1}{k-2}\left[\frac{n-k(x_1-1)}{k-1}\right].
		\end{align*}
		Thus, if $x_1-1> \lceil \frac{n}{k}\rceil$, it follows that $|V_X|<|V_Y|$. On the other hand, if $x_1-1=\lceil \frac{n}{k}\rceil$, then we have $n-k(x_1-1)\le 0$. Furthermore, by Equation~\eqref{eq:x_1greater}, the inequality ($\le$) in our sequence of steps becomes a strict inequality ($<$), ensuring that  $|V_X|<|V_Y|$.  
		In either case, we conclude that $|V_X|< |V_Y|$, as claimed. This completes the proof. 
	\end{itemize}
\end{proof}

We are ready to prove our upper bound on $\trw(F_k(K_n))$ as stated in Theorem~\ref{Theo:twKn}. 
\johnsonk*
\begin{proof}
	Recall that the  width of $(P,\mathcal{V})$ is an upper bound for $\trw(F_k(K_n))$. 
	Let $X:=\{x_1,\dots,x_{k-1}\}$ and $Y:=\{y_1,\dots,y_{k-1}\}$, where 
	\begin{align*}
		{x_i:=\begin{cases}
				\lfloor \frac{n+1}{k}\rfloor & \textrm{if $i=1$,}\\
				n-(k-i)x_1 & \textrm{otherwise,}
		\end{cases}} \qquad \textrm{and}\qquad {y_i:=\begin{cases}
				\lceil\frac{n}{k}\rceil & \textrm{if $i=1$,}\\
				n-(k-i)y_1 & \textrm{otherwise.}
		\end{cases}} 
	\end{align*}
	By Lemmas~\ref{lemma:max-x_i_value}~and~\ref{lemma:possible-bags} it follows that
	$\operatorname{width}(P,\mathcal{V})=\max\{|V_X|,|V_Y|\}-1$. Then, by Proposition~\ref{prop:size-of-bags}, we have that 
	$\operatorname{width}(P,\mathcal{V})$ is equal to
	\begin{align*}
		\max & \left \{ \left \lfloor\frac{n+1}{k} \right \rfloor\binom{n-\lfloor\frac{n+1}{k}\rfloor}{k-1}+\left \lfloor\frac{n+1}{k} \right \rfloor\sum\limits_{i=2}^{k}\binom{(k-i)\lfloor\frac{n+1}{k}\rfloor}{k-i}-\sum\limits_{i=1}^{k-1}\binom{(k-i-1)\lfloor\frac{n+1}{k}\rfloor+1}{k-i}, \right.  \\
		&  \left. \left \lceil \frac{n}{k} \right \rceil \binom{n-\lceil \frac{n}{k} \rceil}{k-1}+\left \lceil \frac{n}{k} \right \rceil\sum\limits_{i=2}^{k}\binom{(k-i)\lceil \frac{n}{k} \rceil}{k-i}-\sum\limits_{i=1}^{k-1}\binom{(k-i-1)\lceil \frac{n}{k} \rceil+1}{k-i} \right \}-1.
	\end{align*}
	Finally, note that if $k$ divides $n$ or $n+1$, then $\lfloor\frac{n+1}{k}\rfloor=\lceil \frac{n}{k}\rceil$.
\end{proof}

We conclude this section by establishing the precise value of our upper bound on $\trw(F_3(K_n))$.
\johnsonthree*
\begin{proof}
	Let $X$ and $Y$ be as defined in the proof of Theorem~\ref{Theo:twKn}.
	Notice that $\lfloor \frac{n+1}{3} \rfloor \le \lceil \frac{n}{3} \rceil$, and if equality holds, our upper bound is the desired value.
	Suppose now that $\lfloor \frac{n+1}{3} \rfloor < \lceil \frac{n}{3} \rceil$, implying $n \equiv 1 \pmod{3}$.
	Let $t := \lfloor \frac{n+1}{3} \rfloor$, so $t+1 = \lceil\frac{n}{3}\rceil$ and $n = 3t+1$.
	Then,  
 	\begin{align*} 
		|V_Y|-|V_X|&= (t+1)\left[\binom{n-t-1}{2}+\binom{t+1}{1}+\binom{0}{0}\right]-\left[\binom{t+2}{2}+\binom{1}{1}\right]\\ &
		\qquad\quad -t\left[\binom{n-t}{2}+\binom{t}{1}+\binom{0}{0}\right]+\left[\binom{t+1}{2}+\binom{1}{1}\right]   \\ 
		& =(t+1)\left[\binom{n-t-1}{2}+t+2\right]-\binom{t+2}{2}-t\left[\binom{n-t}{2}+t+1\right]+\binom{t+1}{2}\\
		& =(t+1)\binom{n-t-1}{2}+(t+1)(t+2)-\binom{t+2}{2}-t\binom{n-t}{2}-t(t+1)+\binom{t+1}{2} \\
		& =(t+1)\binom{n-t-1}{2}-t\binom{n-t}{2}+t+1\\
		&=\binom{n-t-1}{2}-t(n-t-2)+1\\
		&=(n-t-2)\left(\frac{n-t-1}{2}-t\right)+1\\
		&=(n-t-2)\frac{n-3t-1}{2}+1\\
		&=1. 
	\end{align*}
	Hence, $|V_Y| > |V_X|$. Finally, we have
	\begin{align*}
		\trw(F_3(K_n))\le |V_Y|-1&=\left\lceil\frac{n}{3}\right\rceil \left[\binom{n-\left\lceil\frac{n}{3}\right\rceil}{2}+\left\lceil\frac{n}{3}\right\rceil+1\right]-\left[\binom{\left\lceil\frac{n}{3}\right\rceil+1}{2}+1\right]-1 \\
		&=\left\lceil\frac{n}{3}\right\rceil \binom{n-\left\lceil\frac{n}{3}\right\rceil}{2}+\frac{1}{2}\left\lceil\frac{n}{3}\right\rceil\left(\left\lceil\frac{n}{3}\right\rceil+1\right)-2.
	\end{align*}
\end{proof}

\section{Summary of results and conjectures}{\label{sec:SR&C}}

In this paper, we have established tight asymptotic lower and upper bounds on the
treewidth of token graphs for three families of graphs. The first two of
them are special cases of trees (stars and paths), for which we obtained
that the treewidth of their $k$-token graphs, for a fixed integer $k$, is $\Theta(n^{k-1})$.
Additionally, we have shown that the treewidth of the $k$-token graph of
the complete graph $K_n$ (Johnson graph) is, for a fixed integer $k$, equal to
$\Theta(n^{k})$. These results lead us to formulate the following
conjectures:

\begin{conjecture}
	\label{conj:trees}
	If  $2 \le k \le n-2$ is a fixed integer and $G$ is a tree on $n$
	vertices, then
	\[
	\trw(F_k(G)) \in\Theta(n^{k-1}).
	\]
\end{conjecture}

\begin{conjecture}
	\label{conj:complete}
	If  $2 \le k \le n-2$ is a fixed integer, then the tree
	decomposition of the  Johnson graph given in Section
	\ref{sec:fkkn} is optimal, and therefore  
	\[
	\prw(F_{k}(K_{n})) = \trw(F_{k}(K_{n})).
	\] 
	Furthermore, the treewidth of the Johnson graph is equal to the
	upper bound given in Theorem \ref{Theo:twKn}.
\end{conjecture}

\bibliographystyle{abbrv}
\bibliography{References_treewidth}

\end{document}